\documentclass[12pt]{amsart}
\usepackage{graphicx}
\usepackage{amssymb}
\usepackage{amsmath}
\usepackage{amscd}
\usepackage{upgreek}
\usepackage{mathrsfs}
\usepackage{stmaryrd}
\usepackage{longtable}
\usepackage{txfonts}
\usepackage[T1]{fontenc}
\usepackage{eucal}
\usepackage{titletoc}
\usepackage{wrapfig}
\usepackage{float}
\usepackage{xypic}
\usepackage{dsfont}
\usepackage{xcolor}
\usepackage{color}
\usepackage[colorlinks,linkcolor=blue,anchorcolor=blue,
citecolor=blue]{hyperref}
\usepackage{hyperref}
\usepackage[hmarginratio=1:1,top=32mm,columnsep=20pt]{geometry}
	
\allowdisplaybreaks	
\usepackage{mathptmx}
\DeclareMathAlphabet{\mathcal}{OMS}{cmsy}{m}{n}
\DeclareSymbolFont{largesymbols}{OMX}{cmex}{m}{n}

\newtheorem{theorem}{Theorem}[section]
\newtheorem{prop}[theorem]{\rm \textsc{Proposition}}

\newtheorem{lem}[theorem]{\rm \textsc{Lemma}}
\newtheorem{coro}[theorem]{\rm \textsc{Corollary}}

\newtheorem{thm}[theorem]{\it \textsc{Theorem}}

\newtheorem{rem}[theorem]{\rm \textsc{Remark}}
\newtheorem{exam}[theorem]{\rm\textsc{Example}}

\newcommand{\N}{\mathbb{N}}

\newcommand{\F}{\mathbb{F}}

\newcommand{\ra}{\longrightarrow}

\newcommand{\GL}{{\rm GL}}
\newcommand{\SL}{{\rm SL}}
\newcommand{\U}{{\rm U}}
\newcommand{\diag}{{\rm diag}}

\newcommand{\B}{{\mathcal B}}
\newcommand{\D}{{\mathcal D}}
\newcommand{\A}{\mathcal{A}}
\newcommand{\sign}{{\rm sign}}

\begin{document}
\setlength{\oddsidemargin}{0cm}
\setlength{\evensidemargin}{0cm}

\title{\scshape Vector invariant fields of finite classical groups}
\author{\scshape Yin Chen}
\address{School of Mathematics and Statistics, Northeast Normal University, Changchun 130024, P.R. China}
\email{ychen@nenu.edu.cn}

\author{\scshape Zhongming Tang}
\address{Department of Mathematics, Soochow (Suzhou) University, Suzhou 215006, P.R. China}
\email{zmtang@suda.edu.cn}

\date{\today}
\def\shorttitle{Vector invariant fields of finite classical groups}

\begin{abstract}
Let $W$ be an $n$-dimensional vector space over a finite field $\F_q$ of any characteristic and $mW$ denote the direct sum of $m$ copies of $W$.
Let $\F_q[mW]^{\GL(W)}$ and $\F_q(mW)^{\GL(W)}$ denote the vector invariant ring and vector invariant field respectively where $\GL(W)$ acts on $W$ in the standard way and acts on $mW$ diagonally.
We prove that there exists a set of homogeneous invariant polynomials $\{f_{1},f_{2},\ldots,f_{mn}\}\subseteq \F_q[mW]^{\GL(W)}$ such that $\F_q(mW)^{\GL(W)}=\F_q(f_{1},f_{2},\ldots,f_{mn})$. We also prove the same assertions for the special linear groups and  the symplectic groups in any characteristic, and the unitary groups and the orthogonal groups  in odd characteristic.
\end{abstract}


\maketitle
\baselineskip=17pt


\dottedcontents{section}[1.16cm]{}{1.8em}{5pt}
\dottedcontents{subsection}[2.00cm]{}{2.7em}{5pt}
\dottedcontents{subsubsection}[2.86cm]{}{3.4em}{5pt}


\section{Introduction}
\setcounter{equation}{0}
\renewcommand{\theequation}
{1.\arabic{equation}}
\setcounter{theorem}{0}
\renewcommand{\thetheorem}
{1.\arabic{theorem}}

For a finite group $G$ and an $n$-dimensional representation $W$ over a field $\F$,
the invariant ring $\F[W]^G$ and the invariant field $\F(W)^G$ are two main objects of study in the invariant theory of finite groups. The rationality problem for $\F(W)^G$ associated with the name of Emmy Noether, has been studied extensively since Swan's counterexample \cite{Swa1969} appeared. Motivated by connecting the rationality of $\F(W)^G$ to characterization of the structure of $\F[W]^G$, one seeks to find a generating set of polynomial invariants for $\F(W)^G$; see Richman \cite{Ric1990} and Chen-Wehlau \cite{CW2017}. More precisely, one asks whether there exist homogenous polynomials $f_1,f_2,\dots,f_n\in \F[W]^G$ such that $\F(W)^G=\F(f_1,f_2,\dots,f_n)$; see Charnow  \cite{Cha1969}, Kang \cite{Kan2006}, Campbell-Chuai \cite{CC2007}, and Chen-Wehlau \cite{CW2019}. The goal of this paper is to answer this question for certain modular vector invariant fields of finite classical groups.

We let $\F_q$ be a finite field of order $q$ with characteristic $p>0$ and
$W$ be an $n$-dimensional vector space over $\F_{q}$.  Let $\U(W)$ denote the Sylow $p$-subgroup of the group
$G(W)$ where $\SL(W)\leq G(W)\leq \GL(W)$. Consider the dual space $W^*$ of $W$ and the direct sum $mW\oplus dW^*$ of $m$ copies of $W$ and $d$ copies of $W^*$. Let $G$ be one of $\{\U(W),\GL(W),\SL(W)\}$ acting on
$mW\oplus dW^*$ diagonally. The previous paper of the first author \cite{CW2019}, together with \cite{CC2007}, shows that there exist homogenous polynomials $f_1,f_2,\dots,f_{(m+d)n}\in \F_q[mW\oplus dW^*]^G$ such that
the vector invariant field $\F_q(mW\oplus dW^*)^{G}=\F_q(f_1,f_2,\dots,f_{(m+d)n})$  for all cases except when
$md=0$ and $G=\GL(W)$ or $G=\SL(W)$. 
Our proof relied upon some relations among a generating set for the vector invariant ring 
$\F_{q}[W\oplus W^{*}]^{\GL(W)}$; see Chen-Wehlau \cite{CW2017}.
However, since the structure of $\F_{q}[W\oplus W]^{\GL(W)}$ is not well understood, it seems that the method we used in Chen-Wehlau \cite{CW2019} can not be applied directly to the question of the remaining case which asks how to find a generating set of polynomial invariants for $\F_q(mW)^{\GL(W)}$ or $\F_q(mW)^{\SL(W)}$ for $m\in\N^{+}$.

To deal with the remaining case, we need to recall a result due to   Steinberg  that provides a generating set $\{\ell_{ij}/\ell_0\mid 1\leqslant i\leqslant m,1\leqslant j\leqslant n\}$ of rational invariants for $\F_q(mW)^{\GL(W)}$; see Steinberg \cite[Corollary]{Ste1987}.

\begin{thm}[Steinberg]\label{Steinberg1} There exist $mn+1$ homogeneous polynomials $\{\ell_0,\ell_{ij}\mid 1\leq i\leq m,1\leq i\leq n\}\subseteq \F_q[mW]^{\SL(W)}$ such that $\F_q(mW)^{\GL(W)}=\F_q(\ell_{ij}/\ell_0\mid 1\leqslant i\leqslant m,1\leqslant j\leqslant n)$.
\end{thm}

However, the original proof of Steinberg's Theorem was extremely short and seems to be not well-readable to us. The first propose of this paper is to give an elementary proof to Steinberg's Theorem, without going into the theory of algebraic groups but Galois theory and localizations in commutative algebra. Our proof is more understandable than Steinberg's one and further it
provides a sample that how to use Galois theory to find a generating set of polynomials invariants for vector invariant fields; see Section \ref{Section2}.

After giving a proof to  Theorem \ref{Steinberg1} for the case $m\geq n$ in Section \ref{Section2}, we develop 
a useful criterion to detect when an invariant ring is a localized polynomial ring in Section \ref{Section3}. As an application, 
we prove Theorem \ref{Steinberg1} for the case $m<n$.
We also  provide several applications of  Theorem \ref{Steinberg1}. In particular,  we prove the following result whose proof will be separated into Theorem \ref{thm3.3} and Corollary \ref{coro3.4}.


\begin{thm}\label{thm1.2}
Let $m\in\N^{+}$ and $W$ be an $n$-dimensional vector space over a finite field $\F_{q}$. Suppose
 $G\in \{\GL(W),\SL(W)\}$. Then there exist a set of homogeneous invariant polynomials $\{f_{1},f_{2},\dots,f_{mn}\}\subseteq \F_{q}[mW]^{G}$ such that
$\F_{q}(mW)^{G}=\F_{q}(f_{1},f_{2},\dots,f_{mn})$.
\end{thm}

Combining this result with \cite[Theorem 1.2]{CW2019} we obtain

\begin{coro}\label{coro1.3}
Let $m,d\in\N$ and $W$ be an $n$-dimensional vector space over a finite field $\F_{q}$. Suppose
 $G\in \{\GL(W),\SL(W),\U(W)\}$. Then there exist a set of homogeneous invariant polynomials
 $\{f_{1},f_{2},\dots,f_{(m+d)n}\}\subseteq \F_{q}[mW\oplus dW^{*}]^{G}$ such that
$\F_{q}(mW\oplus dW^{*})^{G}=\F_{q}(f_{1},f_{2},\dots,f_{(m+d)n})$. In particular, $\F_{q}[mW\oplus dW^{*}]^{G}$  is
a localized polynomial ring, i.e., there exists an element $f\in \F_{q}[mW\oplus dW^{*}]^{G}$ such that
$\F_{q}[mW\oplus dW^{*}]^{G}[f^{-1}]=\F_{q}[f_{1},f_{2},\dots,f_{(m+d)n}][f^{-1}]$.
\end{coro}

Section \ref{Section4} is devoted to finding a minimal generating set of polynomial invariants for the vector invariant field of
other finite classical groups, such as symplectic, unitary and  orthogonal groups. Let $O(W)$ be the orthogonal group over a finite field $\F_{q}$ of odd characteristic with the standard representation $W$. We prove an analogue of Theorem \ref{thm1.2} for $O(W)$, i.e., we find
$f_{1},f_{2},\dots,f_{mn}\in \F_{q}[mW]^{O(W)}$ such that
$\F_{q}(mW)^{O(W)}=\F_{q}(f_{1},f_{2},\dots,f_{mn})$ where $n=\dim(W)$ and $m\in\N^{+}$.
We also derive similar conclusions for the finite unitary and symplectic groups; see Theorem \ref{thm4.1} for details.

\begin{rem}{\rm Let $G\subseteq \GL(W)$ be a subgroup for which $\F_q(W)^{G}$ is rational over $\F_{q}$.
Then the rationality of $\F_q(mW)^{G}$ can be seen by directly applying the so-called ``No-name Lemma''; see for example Jensen-Ledet-Yui \cite[Section 1.1, page 22]{JLY2002}.
}\end{rem}

\section{Dickson Invariants and Steinberg's Theorem}\label{Section2}
\setcounter{equation}{0}
\renewcommand{\theequation}
{2.\arabic{equation}}
\setcounter{theorem}{0}
\renewcommand{\thetheorem}
{2.\arabic{theorem}}

The main purpose of this section is to give a proof to Theorem  \ref{Steinberg1} in the case $m\geq n$. We need to recall the classical Dickson invariants and extend Steinberg's construction.  Suppose $\F_{q}[W]=\F_{q}[x_{1},x_{2},\dots,x_{n}]$ and consider the following $n\times n$ matrix
$$D_{ni}:=\begin{pmatrix}
     x_{1} &x_{1}^{q}&\cdots&\widehat{x_{1}^{q^{i}}} &\cdots& x_{1}^{q^{n}}  \\
     x_{2} &x_{2}^{q}&\cdots&\widehat{x_{2}^{q^{i}}} &\cdots& x_{2}^{q^{n}}  \\
     \vdots&\vdots&\cdots&\vdots&\cdots&\vdots\\
    x_{n} &x_{n}^{q}&\cdots&\widehat{x_{n}^{q^{i}}} &\cdots& x_{n}^{q^{n}}  \\
\end{pmatrix}$$ where $\widehat{\ast}$ denotes the symbol $\ast$ was deleted.
Define $d_{ni}:=\det(D_{ni})$ and $c_{ni}:=d_{ni}/d_{nn}$ for $0\leqslant i\leqslant n.$ Then
$\F_{q}[W]^{\GL(W)}=\F_{q}[c_{n0},c_{n1},\dots,c_{n,n-1}]$ and $\F_{q}[W]^{\SL(W)}=\F_{q}[d_{nn},c_{n1},\dots,c_{n,n-1}]$
are polynomial algebras over $\F_{q}$; see Dickson \cite{Dic1911} for the original proofs or Wilkerson \cite{Wil1983} for a modern treatment.

Throughout this paper, $^{t}\ast$ stands for the transpose of a matrix (or vector) $\ast$.
Suppose that $\F_{q}(mW)=\F_{q}(x_{ij}\mid 1\leq i\leq m,1\leq j\leq n)$
and we are working over $\F_{q}(mW)$. The constructions will be separated into two subcases.

(1) We first suppose $m\geq n\geq 1$. Extending Steinberg's construction \cite[Section 3]{Ste1987},
for $1\leq i\leq m$, we define $X_{i}$ to be the column vector $^{t}(x_{i1},x_{i2},\dots,x_{in})$ and for $k\in{\mathbb N}$, we define $X_{i}^{q^k}:=~^{t}(x_{i1}^{q^k},x_{i2}^{q^k},\dots,x_{in}^{q^k})$ to be the $q^{k}$-th power of $X_{i}$. Consider the $n\times m$ matrix $L=(X_{1},X_{2},\dots,X_{m})$ and an $n\times n$-submatrix $L_{0}=(X_{1},X_{2},\dots,X_{n})$. For $1\leq j\leq n$ and $k\in{\mathbb N}$, if $i\leq n$, we define
$L_{ij}^{(k)}$ to be the matrix obtained from $L_{0}$ by replacing the $j$-th column $X_{j}$ of  $L_{0}$ by the $q^{k}$-th power  $X_{i}^{q^k}$ of the $i$-th column of $L_{0}$; if $n<i\leq m$, we define
$L_{ij}^{(k)}$ to be the matrix obtained from $L_{0}$ by replacing the $j$-th column $X_{j}$ of  $L_{0}$ by the $i$-th column $X_{i}$ of  $L$. Namely,
$$L_{ij}^{(k)}=\begin{cases}
   (X_{1},\dots,X_{j-1},X_{i}^{q^k},X_{j+1},\dots,X_{n}), & \text{if } i\leq n; \\
 (X_{1},\dots,X_{j-1},X_{i},X_{j+1},\dots,X_{n}), & \text{if } n<i\leq m.
\end{cases}
$$
Let $\ell_{0}=\det(L_{0})$ and $\ell_{ij}^{(k)}=\det(L_{ij}^{(k)})$. Note that when $n<i\leq m$, $\ell_{ij}^{(k)}$ is independent of $k$.
To coincide with Steinberg's notation, we denote $\ell_{ij}^{(1)}$ by $\ell_{ij}$. We observe that
every $\ell_{ij}^{(k)}$ is a $\det$-invariant, i.e., $\sigma(\ell_{ij}^{(k)})=\det(\sigma)\ell_{ij}^{(k)}$
for all $\sigma\in{\rm GL}(W)$. Now Theorem \ref{Steinberg1} states that $\F_{q}(mW)^{\GL(W)}=\F_{q}(\ell_{ij}/\ell_{0}\mid 1\leq i\leq m,1\leq j\leq n)$.

(2) Secondly, we suppose that $m<n$. We add some $q^{k}$-th powers of $X_{m}$ into the $n\times m$-matrix $(X_{1},X_{2},\dots,X_{m})$ such that we may obtain an $n\times n$-matrix $$L=(X_{1},X_{2},\dots,X_{m},X_{m}^{q},\dots,X_{m}^{q^{n-m}}).$$ We consider the following matrix
$$L_{0} = (X_{1},X_{2},\dots,X_{m-1},X_{m}^{q^{n-m}},X_{m},X_{m}^{q},\dots,X_{m}^{q^{n-m-1}})$$
which is obtained from $L$ by moving the last column $X_{m}^{q^{n-m}}$ of $L$ forward to the $m$-th column.
To derive $\ell_{ij}/\ell_{0}$, we use the same construction appeared in the previous case. Namely, for $1\leq i, j\leq n$, we define $L_{ij}$ to be the matrix obtained from $L_{0}$ by replacing the $j$-th column of  $L_{0}$ by the $q$-th power  of the $i$-th column of $L_{0}$.
Let $\ell_{0}=\det(L_{0})$ and $\ell_{ij}=\det(L_{ij})$.
By Dickson's theorem we see that $\det(X_{1},X_{1}^{q},\dots,X_{1}^{q^{n-1}})$ is not zero,  thus $\ell_{0}\neq 0$.
We observe that for $m+1\leq i\leq n$, $\ell_{ij}/\ell_{0}$ is constant.
Thus Theorem \ref{Steinberg1} in this case also could be read to 
$\F_{q}(mW)^{\GL(W)}=\F_{q}(\ell_{ij}/\ell_{0}\mid 1\leq i\leq m,1\leq j\leq n)$.

We make a couple of remarks to explain above constructions. 

\begin{rem}{\rm
In the original construction of the case $m<n$ (\cite[page 704, the second paragraph]{Ste1987}), Steinberg 
took $$L_{0}=L=(X_{1},X_{2},\dots,X_{m},X_{m}^{q},\dots,X_{m}^{q^{n-m}})$$ and obtained 
a matrix $^{t}(\ell_{ij}/\ell_{0})_{n\times n}$ in which all the last $n-m+1$ columns are constant except for the $n$-th column. 
This means that the Steinberg's construction eventually derived that 
$\F_{q}(mW)^{\GL(W)}=\F_{q}(\ell_{ij}/\ell_{0}, \ell_{nj}/\ell_{0}\mid 1\leq i\leq m-1,1\leq j\leq n)$.
There are no essential difference between the Steinberg's original construction and the one mentioned above. 
}\end{rem}

\begin{rem}{\rm
Note that the $\ell_{ij}$ (or $\ell_{0}$) in Theorem \ref{Steinberg1} are different for the case $m\geq n$  and the case $m<n$. 
For example, setting $n=3$, we consider $\F_q(3W)^{{\rm GL}(W)}$ and $\F_q(2W)^{{\rm GL}(W)}$.
By above constructions, the $\ell_{12}$ in $\F_q(3W)^{{\rm GL}(W)}$ is
$\det(X_{1},X_{1}^{q},X_{3})$ as well as  the $\ell_{12}$ in $\F_q(2W)^{{\rm GL}(W)}$
is $\det(X_{1},X_{1}^{q},X_{2})$.
}\end{rem}

\begin{rem}{\rm
For the special case $m=1$, the above-mentioned construction, up to permutation and sign,  may produce the generating set $\{c_{n0},c_{n1},\dots,c_{n,n-1}\}$ for the classical Dickson algebra we have seen at the beginning of this section. Furthermore, Magma calculation \cite{magma} shows that $\ell_{ij}/\ell_{0}$   in Theorem \ref{Steinberg1} might not be polynomial for $m\geq 2$.
}\end{rem}

The following lemma indicates that the proof of Theorem  \ref{Steinberg1} for the case $m\geq n$ could be reduced  to the case $m=n$.

\begin{lem}\label{lem2.4}
Let $K=\F_q(x_{ij}\mid 1\leq i,j\leq n)$. If $m>n$, then $\F_q(mW)^{\GL(W)}=K^{\GL(W)}(\ell_{kj}/\ell_0\mid n+1\leqslant k\leqslant m,1\leqslant j\leqslant n)$.
\end{lem}

\begin{proof} For any integer $k\in\{n+1,\dots, m\}$, the Cramer's rule implies that the non-homogenous linear equations $L_{0}\cdot Y=X_{k}$
has a unique common solution $Y_k=~^{t}(\frac{\ell_{k1}}{\ell_0},\frac{\ell_{k2}}{\ell_0},\dots, \frac{\ell_{kn}}{\ell_0})$. This means that every $x_{kj}$ can be expressed linearly by
$\{\frac{\ell_{k1}}{\ell_0},\frac{\ell_{k2}}{\ell_0},\dots, \frac{\ell_{kn}}{\ell_0}\}$ over $K$ for $1\leq j\leq n$. Hence, $K(x_{k1},x_{k2},\dots,x_{kn})=K(\frac{\ell_{k1}}{\ell_0},\frac{\ell_{k2}}{\ell_0},\dots, \frac{\ell_{kn}}{\ell_0})$. As $\{x_{kj}\mid n+1\leq k\leq m,1\leq j\leq n\}$ is a set of algebraic independent elements over $K$, then
$$\F_q(mW)=K(x_{kj}\mid n+1\leq k\leq m,1\leq j\leq n)=K(\ell_{kj}/\ell_0\mid n+1\leq k\leq m,1\leq j\leq n)$$ is rational over $K$. Note that every $\frac{\ell_{kj}}{\ell_0}$ is $\GL(W)$-invariant.
Thus the statement holds.
\end{proof}

As  a direct consequence, we have

\begin{coro} 
If $n=\dim(W)=1$ and $m\geq 1$, then 
$$\F_{q}(mW)^{\GL(W)}=\F_{q}(x_{11},x_{21},\dots,x_{m1})^{\F_{q}^{\times}}=\F_{q}(x_{11}^{q-1},x_{i1}/x_{11}\mid i=2,\dots,m).$$
\end{coro}

The rest of this section is devoted to proving Theorem \ref{Steinberg1} for the case $m\geq n$.
Throughout we may suppose $m\geq n\geq 2.$

\begin{lem}\label{lem2.6}
For $1\leq i\leq n$, we let $\{c_{ns}^{(i)}=d_{ns}^{(i)}/d_{nn}^{(i)}\mid 0\leq s\leq n-1\}$ denote the Dickson invariants for $\GL(W)$ acting on the polynomial ring $\F_q[x_{i1},x_{i2},\dots,x_{in}]$. Then
$$c_{ns}^{(i)}\in \F_q(\ell_{ij}/\ell_0\mid 1\leqslant i, j\leqslant n)$$
for all $0\leq s\leq n-1$.
\end{lem}

\begin{proof}
Note that $\ell_{0}=\det(L_{0})\neq0$. For each $i\in\{1,2,\dots,n\}$ and $k\in\N$,
it follows from the Cramer's rule that the non-homogenous linear equations $L_{0}\cdot Y=X_{i}^{q^k}$
has a unique common solution $Y_{i}^{(k)}=~^{t}(\ell_{i1}^{(k)}/\ell_{0},\ell_{i2}^{(k)}/\ell_{0},\dots,\ell_{in}^{(k)}/\ell_{0})$. Thus
$$
L_0\cdot\left(
\begin{array}{cccccc}
\ell_{i1}^{(0)}/\ell_0&\cdots&\ell_{i1}^{(s-1)}/\ell_0&\ell_{i1}^{(s+1)}/\ell_0&\cdots&\ell_{i1}^{(n)}/\ell_0\\
\ell_{i2}^{(0)}/\ell_0&\cdots&\ell_{i2}^{(s-1)}/\ell_0&\ell_{i2}^{(s+1)}/\ell_0&\cdots&\ell_{i2}^{(n)}/\ell_0\\
\vdots&\cdots&\vdots&\vdots&\cdots&\vdots\\
\ell_{in}^{(0)}/\ell_0&\cdots&\ell_{in}^{(s-1)}/\ell_0&\ell_{in}^{(s+1)}/\ell_0&\cdots&\ell_{in}^{(n)}/\ell_0
\end{array}\right)=(X_i,X_i^q,\dots, X_i^{q^{s-1}},X_i^{q^{s+1}},\dots, X_i^{q^n}).
$$
Taking determinant we see that
\begin{equation}
\label{eq2.1}
\ell_0\cdot\det\left(
\begin{array}{cccccc}
\ell_{i1}^{(0)}/\ell_0&\cdots&\ell_{i1}^{(s-1)}/\ell_0&\ell_{i1}^{(s+1)}/\ell_0&\cdots&\ell_{i1}^{(n)}/\ell_0\\
\ell_{i2}^{(0)}/\ell_0&\cdots&\ell_{i2}^{(s-1)}/\ell_0&\ell_{i2}^{(s+1)}/\ell_0&\cdots&\ell_{i2}^{(n)}/\ell_0\\
\vdots&\cdots&\vdots&\vdots&\cdots&\vdots\\
\ell_{in}^{(0)}/\ell_0&\cdots&\ell_{in}^{(s-1)}/\ell_0&\ell_{in}^{(s+1)}/\ell_0&\cdots&\ell_{in}^{(n)}/\ell_0
\end{array}\right)=d_{ns}^{(i)}.
\end{equation}
Hence, $c_{ns}^{(i)}=\frac{d_{ns}^{(i)}/\ell_0}{d_{nn}^{(i)}/\ell_0}\in \F_q(\ell_{ij}^{(k)}/\ell_0\mid 1\leq i,j\leq n,0\leq k\leq n)$ for all $s\in\{0,1,\dots,n-1\}$. It remains to show that
every $\ell_{ij}^{(k)}/\ell_0\in \F_q(\ell_{ij}/\ell_0\mid 1\leqslant i,j\leqslant n).$
To do this, we consider the following equality
\begin{equation}\label{eq2.2}
L_0\cdot\left(
\begin{array}{cccccc}
\ell_{11}^{(k)}/\ell_0&\ell_{21}^{(k)}/\ell_0&\cdots&\ell_{n1}^{(k)}/\ell_0\\
\ell_{12}^{(k)}/\ell_0&\ell_{22}^{(k)}/\ell_0&\cdots&\ell_{n2}^{(k)}/\ell_0\\
\vdots&\vdots&\cdots&\vdots\\
\ell_{1n}^{(k)}/\ell_0&\ell_{2n}^{(k)}/\ell_0&\cdots&\ell_{nn}^{(k)}/\ell_0
\end{array}\right)=(X_1^{q^k},X_2^{q^k},\dots, X_n^{q^k})=:L_0^{(q^k)}.
\end{equation}
Substituting $x_{ij}^{q^t}$ for $x_{ij}$ in this equality and assuming $k=1$, we obtain
\begin{equation}\label{eq2.3}
L_0^{(q^t)}\cdot\left((\ell_{ij}/\ell_0)^{q^t}\right)_{n\times n}
=L_0^{(q^{t+1})}.
\end{equation}
Combining (\ref{eq2.2}) and (\ref{eq2.3}) we see that
\begin{eqnarray}
\left(\ell_{ij}^{(k)}/\ell_0\right)_{n\times n}
&=&L_0^{-1}\cdot L_0^{(q^k)}\nonumber\\
&=&(L_0^{-1}\cdot L_0^{(q)})((L_0^{(q)})^{-1}\cdot L_0^{(q^2)})\cdots((L_0^{(q^{k-1})})^{-1}\cdot L_0^{(q^k)})\nonumber\\
&=&\left(\ell_{ij}/\ell_0\right)_{n\times n}\cdot\left((\ell_{ij}/\ell_0)^q\right)_{n\times n}\cdots\left((\ell_{ij}/\ell_0)^{q^{k-1}}\right)_{n\times n}.\label{eq2.4}
\end{eqnarray}
Hence, every $\ell_{ij}^{(k)}/\ell_0$ can be expressed polynomially by elements in
$\{\ell_{ij}/\ell_0\mid 1\leqslant i,j\leqslant n\}$.
This completes the proof.
\end{proof}

\begin{lem} \label{lem2.7}
Let $\Sigma_{n}$ denote the symmetric group of degree $n$. Then
\begin{equation}
\label{ }
\ell_{0}^{q-1+n}=\sum_{\sigma={1~2~\dots~n\choose k_{1}k_{2}\cdots k_{n}} \in \Sigma_{n}} (-1)^{\sign(\sigma)}\ell_{1k_{1}}\ell_{2k_{2}}\cdots\ell_{nk_{n}}.
\end{equation}
\end{lem}

\begin{proof}
For each $i\in\{1,2,\dots,n\}$, the Cramer's rule implies that $L_{0}\cdot ~^{t}(\frac{\ell_{i1}}{\ell_{0}},\frac{\ell_{i2}}{\ell_{0}},\dots,\frac{\ell_{in}}{\ell_{0}})=X_{i}^{q}$.
Then
\begin{equation}
\label{ }
L_0\cdot\begin{pmatrix}
   \ell_{11}/\ell_{0}   &   \ell_{21}/\ell_{0} &\cdots& \ell_{n1}/\ell_{0}   \\
   \ell_{12}/\ell_{0}   &  \ell_{22}/\ell_{0} &\cdots& \ell_{n2}/\ell_{0}   \\
   \vdots&\vdots&\vdots&\vdots\\
    \ell_{1n}/\ell_{0}   &   \ell_{2n}/\ell_{0} &\cdots& \ell_{nn}/\ell_{0}   \\
\end{pmatrix}=(X_{1}^{q},X_{2}^{q},\dots,X_{n}^{q}).
\end{equation}
Taking determinants for the above matrices we see that $\ell_{0}\cdot \det(\ell_{ij}/\ell_{0})_{n\times n}=\ell_{0}^{q}.$
Hence $$\ell_{0}^{q-1+n}=\det(\ell_{ij})_{n\times n}=\sum_{\sigma={1~2~\dots~n\choose k_{1}k_{2}\cdots k_{n}} \in \Sigma_{n}} (-1)^{\sign(\sigma)}\ell_{1k_{1}}\ell_{2k_{2}}\cdots\ell_{nk_{n}},$$
as desired.
\end{proof}

We are ready to prove Theorem \ref{Steinberg1} for the case $m\geq n$.

\begin{proof}[Proof of Theorem \ref{Steinberg1} for the case $m\geqslant n$] Note that Lemma \ref{lem2.4} reduces the proof to the case $m=n$.
Let $\widetilde{G}$ be the direct product of $n$ copies of $\GL(W)$. By Kemper  \cite[Proposition 16]{Kem1996}, we see that $\F_q[nW]^{\widetilde{G}}$ is a polynomial algebra over $\F_{q}$, generated by $\{c_{nj}^{(i)}\mid 1\leq i\leq n,0\leqslant j\leqslant n-1\}$. 
Thus $\F_q(nW)^{\widetilde{G}}=\F_q(c_{n0}^{(i)},c_{n1}^{(i)},\dots,c_{n,n-1}^{(i)}\mid 1\leq i\leq n)$. Let
$E$ be the subfield of $\F_q(nW)$ generated by $\{\ell_{ij}/\ell_0\mid 1\leq i,j\leq n\}$ over $\F_q$. By Lemma \ref{lem2.6} we see that $\F_q(nW)^{\widetilde{G}}$ is contained in $E$. Let $H$ be the subgroup of $\widetilde{G}$ consisting of invertible matrices that fix every element in $E$.
Artin's theorem implies that $\F_q(nW)$ is Galois over $\F_q(nW)^{\widetilde{G}}$ with the Galois group $\widetilde{G}$. Thus $\F_q(nW)$ is also Galois over $E$ with the Galois group $H$. Now we have the following situation:
$$
\xymatrix{
\F_q(nW)^{\widetilde{G}} ~\ar@{--}[d] \ar@{^{(}->}[r] &~E~ \ar@{--}[d] \ar@{^{(}->}[r]& ~\F_q(nW)^{\GL(W)}~ \ar@{--}[d] \ar@{^{(}->}[r]&  ~\F_q(nW) \ar@{--}[d]\\
\widetilde{G}~& ~H~\ar@{_{(}->}[l]  &~\GL(W)~ \ar@{_{(}->}[l] & ~1 \ar@{_{(}->}[l]
}
$$
By Galois theory we see that $\F_q(nW)^{\GL(W)}=E$ if and only if $H=\GL(W)$. Thus it remains to show that
$H\subseteq \GL(W)$. For any $\sigma=\diag\{\sigma_1,\sigma_2,\dots,\sigma_n\}\in H$ where each $\sigma_j\in \GL(W)$, it is sufficient to show that $\sigma_1=\sigma_2=\cdots=\sigma_n$.
Let $\tau=\diag\{\sigma_1,\sigma_1,\dots,\sigma_1\}$.
As in the proof of Lemma \ref{lem2.6} we see that $\eta_{s}:=\det(X_{1},X_{s},X_{1}^{q},\dots,X_{1}^{q^{n-2}})=\ell_{0}\cdot f_{s}$
for each $s\in\{2,3,\dots,n\}$, where $f_{s}$ denotes a polynomial in elements of $\{\ell_{ij}^{(k)}/\ell_{0}\mid 1\leqslant i,j,k\leqslant n\}$. The fact that every $\ell_{ij}^{(k)}/\ell_{0}$ belongs to the field $\F_{q}(\ell_{ij}/\ell_{0}\mid 1\leqslant i,j\leqslant n)$, together with Lemma \ref{lem2.7}, implies that $\eta_{s}^{q-1}\in E.$
Thus
$\eta_s^{q-1}=\sigma\cdot(\tau^{-1}\cdot\eta_s^{q-1})=(\sigma\tau^{-1})\cdot\eta_s^{q-1}=
\diag\{1,\sigma_2\sigma_1^{-1},\dots,
\sigma_n\sigma_1^{-1}\}\cdot\eta_s^{q-1}$, i.e.,
$$\det(X_{1},X_{s},X_{1}^{q},\dots,X_{1}^{q^{n-2}})^{q-1}=\det(X_{1},(\sigma_s\sigma_1^{-1})\cdot X_{s},X_{1}^{q},\dots,X_{1}^{q^{n-2}})^{q-1}.$$
Hence there exists an element $a\in\F_{q}^{\times}$ such that
\begin{equation}
\label{eq2.7}
\det(X_{1},X_{s},X_{1}^{q},\dots,X_{1}^{q^{n-2}})=a\cdot\det(X_{1},(\sigma_s\sigma_1^{-1})\cdot X_{s},X_{1}^{q},\dots,X_{1}^{q^{n-2}})
\end{equation}
Taking $x_{s1}=x_{s2}=\dots=x_{sn}=1$ in (\ref{eq2.7}) we see that $a=1$. Thus it follows from (\ref{eq2.7}) that
\begin{equation}
\label{eq2.8}
\det(X_{1},X_{s}-(\sigma_s\sigma_1^{-1})\cdot X_{s},X_{1}^{q},\dots,X_{1}^{q^{n-2}})=0.
\end{equation}
Consider the $n\times (n-1)$-matrix $A:=(X_{1},X_{1}^{q},\dots,X_{1}^{q^{n-2}})$ and let $A_{k}$ denote the $(n-1)\times (n-1)$ matrix obtained from $A$ by deleting the $k$-th row  where $k=1,2,\dots,n$. By the Jacobian criterion (see Benson  \cite[Proposition 5.4.2]{Ben1993}) we see that $\det(A_{1}),\det(A_{2}),\dots,\det(A_{n})$ are algebraically independent over $\F_{q}(x_{s1},x_{s2},\dots,x_{sn})$. Hence, the Laplace expansion along the second column in the determinant of the left hand side of (\ref{eq2.8}), implies that $X_{s}-(\sigma_s\sigma_1^{-1})\cdot X_{s}=0$, i.e., $\sigma_s\cdot\sigma_1^{-1}=I_{n}$, the identity map, for all $s\in\{2,3,\dots,n\}$. Therefore,
$\sigma_1=\sigma_2=\cdots=\sigma_n$ and $\sigma\in \GL(W)$, as required.
\end{proof}

\section{Localized Polynomial Rings} \label{Section3}
\setcounter{equation}{0}
\renewcommand{\theequation}
{3.\arabic{equation}}
\setcounter{theorem}{0}
\renewcommand{\thetheorem}
{3.\arabic{theorem}}

In this section we first use Theorem \ref{Steinberg1} in the special case $m=n$ to give a proof of  Theorem \ref{Steinberg1} for the case $m<n$; and then as an application, we give a proof of Theorem \ref{thm1.2} which consists of Theorem \ref{thm3.3} and Corollary \ref{coro3.4}.
To do this, we need to detect whether the vector invariant ring
$\F_{q}[nW]^{\GL(W)}$ is a localized polynomial ring. The following general criterion will be useful.

\begin{prop}\label{prop3.1}
Let $V$ be a faithful $k$-dimensional representation of a finite group $H$ over a field $\F$ and $G\subseteq H$ be a subgroup. 
Suppose there exist $f_{1},f_{2},\dots,f_{k}\in\F(V)^{G}$ such that $\F(V)^{G}=\F(f_{1},f_{2},\dots,f_{k})$ and there exists a homogenous polynomial $f\in \F[f_{1},f_{2},\dots,f_{k}]\cap \F[V]^{G}$ such that
$\F[f_{1},f_{2},\dots,f_{k}][f^{-1}]\subseteq \F[V]^{G}[f^{-1}]$. 
\begin{enumerate}
  \item If $\F[V]^{G}[f^{-1}]$ is integral over $\F[f_{1},f_{2},\dots,f_{k}][f^{-1}]$,
then $$\F[V]^{G}[f^{-1}]=\F[f_{1},f_{2},\dots,f_{k}][f^{-1}].$$
  \item If  $\F[V]^{H}\subseteq \F[f_{1},f_{2},\dots,f_{k}][f^{-1}]$,
  then $\F[V]^{G}[f^{-1}]=\F[f_{1},f_{2},\dots,f_{k}][f^{-1}].$
\end{enumerate}
\end{prop}

\begin{proof}
(1) 
Since the invariant field $\F(V)^{G}$ is purely transcendental over $\F$, $\{f_{1},f_{2},\dots,f_{k}\}$ is algebraically independent over $\F$. Thus $\F[f_{1},f_{2},\dots,f_{k}]$ is a polynomial subalgebra of $\F(V)^{G}$. From this fact we see that $\F[f_{1},f_{2},\dots,f_{k}][f^{-1}]$ is factorial and so is integrally closed. Note that
the field of fractions of  $\F[f_{1},f_{2},\dots,f_{k}][f^{-1}]$ is $\F(V)^{G}$ which contains $\F[V]^{G}[f^{-1}]$.
As $\F[V]^{G}[f^{-1}]\supseteq\F[f_{1},f_{2},\dots,f_{k}][f^{-1}]$ is integral, we have $\F[V]^{G}[f^{-1}]=\F[f_{1},f_{2},\dots,f_{k}][f^{-1}]$.

(2) We may regard $\F[V]^{H}$ as an $\F$-subalgebra of $\F[V]^{G}$. As $H$ is a finite group, $\F[V]$ is integral over $\F[V]^{H}$, and so is $\F[V]^{G}$. Thus $\F[V]^{G}[f^{-1}]$ is integral over 
$\F[V]^{H}[f^{-1}]$. Since $\F[f_{1},f_{2},\dots,f_{k}][f^{-1}]$ contains $\F[V]^{H}$, it also contains
$\F[V]^{H}[f^{-1}]$. Hence, $\F[V]^{G}[f^{-1}]$ is integral over $\F[f_{1},f_{2},\dots,f_{k}][f^{-1}]$. Now the first statement applies. 
\end{proof}

\begin{prop}\label{prop3.2}
 Let $\dim(W)=n\geq 2$ and let $r$ be the minimal positive integer such that $r(q-1)-n\geq 0$. Define $\ell:=\ell_{0}^{r(q-1)-n}\prod_{i=1}^{n}d_{nn}^{(i)}$, where $\ell_{0}=\det(X_{1},X_{2},\dots,X_{n})$. Then
$$\F_{q}[nW]^{{\rm GL}(W)}[\ell^{-1}]=\F_{q}[\ell_{ij}/\ell_{0}\mid 1\leq i,j\leq n][\ell^{-1}].$$
\end{prop}

\begin{proof} Previously, we have proved that $\F_{q}(nW)^{\GL(W)}=\F_{q}(\ell_{ij}/\ell_{0}\mid 1\leq i,j\leq n)$. 
 Let $B$ denote the polynomial subalgebra of $\F_{q}(nW)^{\GL(W)}$, generated by $\{\ell_{ij}/\ell_{0}\mid 1\leq i,j\leq n\}$ over $\F_{q}$.
Let $H$ be the direct product of $n$ copies of $\GL(W)$ acting on $nW$ diagonally.
Then ${\rm GL}(W)$ can be viewed as a subgroup of $H$.

To see that $\ell\in B\cap \F_{q}[nW]^{\GL(W)}$,  we note that $\ell_{0}$ and every $d_{nn}^{(i)}$ are $\det$-invariants, thus
$\ell\in \F_{q}[nW]^{\GL(W)}$. Moreover, it follows from (\ref{eq2.1}) and (\ref{eq2.4}) that every $d_{nn}^{(i)}/\ell_{0}\in B$.
By Lemma \ref{lem2.7} we see that $\ell_{0}^{q-1}\in B$. Thus
\begin{equation}
\label{eq3.1}
\ell=\ell_{0}^{r(q-1)}\cdot \prod_{i=1}^{n}\frac{d_{nn}^{(i)}}{\ell_{0}}\in B.
\end{equation}
Hence, $\ell\in B\cap \F_{q}[nW]^{\GL(W)}$.

For all $1\leq i,j\leq n$, we have
$$\frac{\ell_{ij}}{\ell_{0}}=\frac{\ell_{ij}\cdot \ell_{0}^{r(q-1)-n-1}\prod_{i=1}^{n}d_{nn}^{(i)}}{\ell_{0}\cdot \ell_{0}^{r(q-1)-n-1}\prod_{i=1}^{n}d_{nn}^{(i)}}=\frac{\ell_{ij}\cdot \ell_{0}^{r(q-1)-n-1}\prod_{i=1}^{n}d_{nn}^{(i)}}{\ell}\in \F_{q}[nW]^{{\rm GL}(W)}[\ell^{-1}].$$
Thus $B[\ell^{-1}]\subseteq \F_{q}[nW]^{{\rm GL}(W)}[\ell^{-1}].$
By the second statement of Proposition \ref{prop3.1}, it suffices to show that 
\begin{equation}\tag{$\dag$}
\label{dag}
\F_{q}[nW]^{H}\subseteq
B[\ell^{-1}].
\end{equation}
Recall that $\F_{q}[nW]^{H}=\F_{q}[c_{ns}^{(i)}\mid 1\leq i\leq n,0\leq s\leq n-1]$ where $c_{ns}^{(i)}=d_{ns}^{(i)}/d_{nn}^{(i)}$
defined as in Lemma \ref{lem2.6}. 
By (\ref{eq2.1}) and (\ref{eq2.4}) we see that every $d_{ns}^{(i)}/\ell_{0}\in B$.
Thus 
\begin{equation}
\label{ }
\ell_{0}^{r(q-1)-n}\cdot d_{ns}^{(i)}\cdot \prod_{i\neq j=1}^{n}d_{nn}^{(j)}=\ell_{0}^{r(q-1)}\cdot \frac{d_{ns}^{(i)}}{\ell_{0}}\cdot \prod_{i\neq j=1}^{n}\frac{d_{nn}^{(j)}}{\ell_{0}}\in B.
\end{equation}
Therefore, for any $1\leq i\leq n,0\leq s\leq n-1$, we have
$$c_{ns}^{(i)}=\frac{d_{ns}^{(i)}}{d_{nn}^{(i)}}=\frac{\ell_{0}^{r(q-1)-n}\cdot d_{ns}^{(i)}\cdot \prod_{i\neq j=1}^{n}d_{nn}^{(j)}}{\ell_{0}^{r(q-1)-n}\cdot  \prod_{j=1}^{n}d_{nn}^{(j)}}=\frac{\ell_{0}^{r(q-1)-n}\cdot d_{ns}^{(i)}\cdot \prod_{i\neq j=1}^{n}d_{nn}^{(j)}}{\ell}\in B[\ell^{-1}].$$
This proves (\ref{dag}), and thus the proof is completed.
\end{proof}

Now we are ready to prove Theorem \ref{Steinberg1} for the case $m<n$.

\begin{proof}[Proof of Theorem \ref{Steinberg1} for the case $m<n$]
We assume that $m<n$ and consider the $\F_{q}$-algebra homomorphism 
$\pi:\F_{q}[X_1,X_2,\dots, X_n]\longrightarrow \F_{q}[X_1,X_2,\dots, X_m]$ defined by fixing $X_{1},\dots,X_{m-1}$,
carrying $X_{m}$ to $X_{m}^{q^{n-m}}$, and carrying $X_{m+k}$ to $X_{m}^{q^{k-1}}$ for $k\in\{1,2,\dots,n-m\}$.
Clearly, the map $\pi$ commutes with the action of $\GL(W)$ and it restricts to a surjective $\F_{q}$-algebra homomorphism 
from $\F_{q}[X_1,X_2,\dots, X_n]^{G}$ to $\F_{q}[X_1,X_2,\dots, X_m]^{G}$ for any subgroup $G\subseteq\GL(W)$.

For any $f\in \F_{q}[mW]^{\GL(W)}=\F_{q}[X_1,X_2,\dots, X_m]^{\GL(W)}$, there exists $f'\in \F_{q}[nW]^{\GL(W)}$
such that $f=\pi(f')$. As we have proved in Proposition \ref{prop3.2} that
\begin{equation}
\label{eq3.4}
\F_{q}[nW]^{{\rm GL}(W)}[\ell^{-1}]=\F_{q}[\ell_{ij}/\ell_{0}\mid 1\leq i,j\leq n][\ell^{-1}],
\end{equation}
where $\ell=\ell_{0}^{q-2}\prod_{i=1}^{n}d_{nn}^{(i)}$. Thus there exists a polynomial $P$ in $n^{2}$ variables and an integer $k$ such that $$f'=\frac{P(\ell_{11}/\ell_{0},\dots,\ell_{1n}/\ell_{0},\ell_{21}/\ell_{0},\dots,\ell_{nn}/\ell_{0})}{\ell^{k}}$$
which implies that there exist an integer $d$ and a polynomial $P'$ in $n^{2}+1$ variables (induced by $P$) such that 
$$\ell_{0}^{d}\cdot f'\cdot\ell^{k}=P'(\ell_{11},\dots,\ell_{1n},\ell_{21},\dots,\ell_{nn},\ell_{0}).$$
A direct computation shows that $\pi(\ell_{0})$ and $\pi(\ell)$ are not zero. Hence,
$$f=\pi(f')=\frac{P(\pi(\ell_{11})/\pi(\ell_{0}),\dots,\pi(\ell_{1n})/\pi(\ell_{0}),\pi(\ell_{21})/\pi(\ell_{0}),\dots,\pi(\ell_{nn})/\pi(\ell_{0}))}{\pi(\ell)^{k}}$$
which belongs to $M:=\F_{q}[\ell_{ij}/\ell_{0}\mid 1\leq i\leq m,1\leq j\leq n][\pi(\ell)^{-1}]$, where
$\ell_{ij}/\ell_{0}$ of $M$ is defined as in Theorem  \ref{Steinberg1}  for the case $m<n$ in Section 2. This means that
\begin{eqnarray*}
\F_{q}(mW)^{\GL(W)}&=&\textrm{Quot}(\F_{q}[mW]^{\GL(W)}[\pi(\ell)^{-1}]) \\
&=&\textrm{Quot}(M)\\
&=&\F_{q}(\ell_{ij}/\ell_{0}\mid 1\leq i\leq m,1\leq j\leq n)
\end{eqnarray*}
where $\textrm{Quot}(-)$ denotes the field of fractions of $-$, and the last equality follows from (\ref{eq3.1}).
Therefore, the proof of Theorem \ref{Steinberg1} is completed.
\end{proof}

We also provide an application of Theorem \ref{Steinberg1}.

\begin{thm}\label{thm3.3}
Let $\B=\{\ell_{0}^{q-2}\ell_{ij}\mid 1\leqslant i\leqslant m,1\leqslant j\leqslant n\}$,
$b\in \{\ell_{0}^{q-2}\ell_{ij}\mid 1\leqslant i\leqslant \min\{m,n\},1\leqslant j\leqslant n\}$ be any element and $\B'=\B\setminus \{b\}$. Then
$\F_{q}(mW)^{\GL(W)}$ is generated by $\{\ell_{0}^{q-1}\}\cup\B'$ over $\F_{q}$.
\end{thm}

\begin{proof} Let $\A=\{\ell_{ij}/\ell_{0}\mid 1\leqslant i\leqslant m,1\leqslant j\leqslant n\}$.
We have proved that $\F_{q}(mW)^{\GL(W)}$ is generated by $\A$ over $\F_{q}$.
Since $\ell_{0}^{q-1}$ is a $\GL(W)$-invariant, it follows that $\F_{q}(mW)^{\GL(W)}$ is generated by $\A\cup\{\ell_{0}^{q-1}\}$ over $\F_{q}$. Lemma \ref{lem2.7} implies that $b/\ell_{0}^{q-1}$ can be expressed rationally by elements in $\A':=(\A\cup\{\ell_{0}^{q-1}\})\setminus \{b/\ell_{0}^{q-1}\}$. Thus $\F_{q}(mW)^{\GL(W)}$ is generated by $\A'$ over $\F_{q}$.
Let $E$ be the subfield of $\F_{q}(mW)$ generated by $\{\ell_{0}^{q-1}\}\cup\B'$ over $\F_{q}$. Clearly,
$E\subseteq \F_{q}(mW)^{\GL(W)}$. Note that every $\ell_{ij}/\ell_{0}=\ell_{0}^{q-2}\ell_{ij}/\ell_{0}^{q-1}$, thus
each element of $\A'$ is contained in $E$. Therefore, $\F_{q}(mW)^{\GL(W)}=E$.
\end{proof}

\begin{coro}\label{coro3.4}
Let $\D=\{\ell_{ij}\mid 1\leqslant i\leqslant m,1\leqslant j\leqslant n\}$,
$d\in \{\ell_{ij}\mid 1\leqslant i\leqslant \min\{m,n\}, 1\leqslant j\leqslant n\}$ be any element and $\D'=\D\setminus \{d\}$. Then
$\F_{q}(mW)^{\SL(W)}$ is generated by $\{\ell_{0}\}\cup\D'$ over $\F_{q}$.
\end{coro}

\begin{proof}
Let $K=\F_{q}(mW)^{\GL(W)}$ and $E$ denote the subfield of $\F_{q}(mW)$ generated by $\{\ell_{0}\}\cup\D'$ over $\F_{q}$. As $\ell_0$ and all $\ell_{ij}$ are $\SL(W)$-invariants, it follows from Theorem \ref{thm3.3} that $K\subset K(\ell_0)=E\subseteq \F_{q}(mW)^{\SL(W)}$. Let $e=\ell_0^{q-1}$. Then $e\in K$ and so $\ell_{0}$ is an algebraic element over $K$. 
Suppose $f(z)=z^{k}+a_{1}z^{k-1}+\dots+a_{k-1}z+a_{k}\in K[z]$ is the minimal polynomial of $\ell_{0}$ over $K$.
Then $\ell_{0}^{k}+a_{1}\ell_{0}^{k-1}+\dots+a_{k-1}\ell_{0}=-a_{k}\in K$. For any $\sigma\in\GL(W)$, setting $b=\det(\sigma)$, we obtain
$$\ell_{0}^{k}b^{k}+a_{1}\ell_{0}^{k-1}b^{k-1}+\dots+a_{k-1}\ell_{0}b=-a_{k}\in K.$$
As $b$ may run over  $\F_{q}^{\times}$, this means that the polynomial 
$$g(z):=\ell_{0}^{k}z^{k}+a_{1}\ell_{0}^{k-1}z^{k-1}+\dots+a_{k-1}\ell_{0}z+a_{k}\in E[z]$$
has at least $q-1$ distinct roots. Thus $k\geq q-1$. On the other hand, the fact that $\ell_{0}^{q-1}-e=0$
implies that $k\leq q-1$. Hence, $k=q-1$. Furthermore,   $[E:K]=[K(\ell_0):K]=q-1=[\GL(W):\SL(W)]=[\F_{q}(mW)^{\SL(W)}:K]$.
Hence, $\F_{q}(mW)^{\SL(W)}=E$, completing the proof.
\end{proof}

\begin{proof}[Proof of Corollary \ref{coro1.3}]
First of all, we note that the all two statements in Corollary \ref{coro1.3} are valid for the group $G=\U(W)$; see Campbell-Chuai \cite[Theorem 2.4]{CC2007}.

Now we suppose $G\in \{{\rm GL}(W),{\rm SL}(W)\}$. It follows from Bonnaf\'e-Kemper \cite[Example 1.2]{BK2011} that Theorem \ref{thm1.2} also holds if we replace $mW$ by $mW^{*}$ for
$G\in \{\GL(W),\SL(W)\}$. Combining this fact, Theorem \ref{thm1.2} and \cite[Theorem 1.2]{CW2019}
we see that the first statement holds. For the second statement, assume that $\F_{q}[mW\oplus dW^{*}]^{G}$ is generated by
$\{g_{1},g_{2},\dots,g_{k}\}$ over $\F_{q}$. Then there exist  $f,h_{i}\in \F_{q}[f_{1},f_{2},\dots,f_{(m+d)n}]$ such that
 $g_{i}=h_{i}/f$ for all $1\leqslant i\leqslant k$. Hence,
 $$\F_{q}[mW\oplus dW^{*}]^{G}\subseteq \F_{q}[h_{1},h_{2},\dots,h_{k}][f^{-1}]\subseteq 
 \F_{q}[f_{1},f_{2},\dots,f_{(m+d)n}][f^{-1}]$$
which implies $\F_{q}[mW\oplus dW^{*}]^{G}[f^{-1}]=\F_{q}[f_{1},f_{2},\dots,f_{(m+d)n}][f^{-1}]$.
\end{proof}

Usually, it is not easy to find an explicit polynomial $f\in \F_{q}[mW\oplus dW^{*}]^{G}$ such that $\F_{q}[mW\oplus dW^{*}]^{G}[f^{-1}]=\F_{q}[f_{1},f_{2},\dots,f_{(m+d)n}][f^{-1}]$; see for example \cite[Section 3]{CC2007} and \cite[Proposition 9]{CW2017}.  We close this section with an example for which $G=\U(W)$, $n=2$ and $m,d\in\N^{+}$.

\begin{exam}{\rm Consider $\F_{q}[mW\oplus dW^{*}]\cong\F_{q}[(\oplus_{j=1}^{m} W_{j})\oplus (\oplus_{k=1}^{d} W_{k}^{*})]$, where each $W_{j}\cong W$ and each $W_{k}^{*}\cong W^{*}$ as $\U(W)$-modules. Suppose $n=2$ and $m,d\in\N^{+}$.
Let $\F_{q}[W_{j}]^{\U(W)}=\F_{q}[f_{j1},f_{j2}]$ and $\F_{q}[W_{k}^{*}]^{\U(W)}=\F_{q}[f_{k1}^{*},f_{k2}^{*}]$ denote the Mui's invariants for $1\leqslant j\leqslant m$ and $1\leqslant k\leqslant d$. Then
$$\F_{q}[mW\oplus dW^{*}]^{\widetilde{\U}(W)}=\F_{q}[f_{j1},f_{j2},f_{k1}^{*},f_{k2}^{*}\mid 1\leqslant j\leqslant m,1\leqslant k\leqslant d]$$
where $\widetilde{\U}(W)$ denotes the direct product of $m+d$ copies of $\U(W)$.
We have seen in \cite[Theorem 3.5]{CW2019} that $\F_{q}(mW\oplus dW^{*})^{\U(W)}$ is minimally generated by
$$\A:=\{f_{11},f_{11}^{*},f_{12}^{*},u_{10}\}\cup\{f_{j1},u_{j0}\mid 2\leqslant j\leqslant m\} \cup\{f_{k1}^{*},v_{k0}\mid 2\leqslant k\leqslant d\},$$
see \cite[Section 3]{CW2019} for the definitions of $u_{j0}$ and $v_{k0}$.
Let $f:=f_{11}\cdot f_{11}^{*}$. To show  $$\F_{q}[mW\oplus dW^{*}]^{\U(W)}_{(f)}=\F_{q}[\A][f^{-1}],$$ it suffices to show that
$f_{12},\dots,f_{m2},f_{22}^{*},\dots,f_{d2}^{*}\in \F_{q}[\A][f^{-1}].$
For $1\leqslant j\leqslant m$, there exists a relation in $\F_{q}[W_{j}\oplus W_{1}^{*}]^{\U(W)}$:
$$u_{j0}^{q}-(f_{j1}f_{11}^{*})^{q-1}u_{j0}-f_{j1}^{q}f_{12}^{*}-f_{11}^{*q}f_{j2}=0;$$
see \cite[Theorem 2.4]{BK2011}. Thus all $f_{j2}\in \F_{q}[\A][f^{-1}]$. Similarly, considering  $\F_{q}[W_{1}\oplus W_{k}^{*}]^{\U(W)}$, we have another relation: $v_{k0}^{q}-(f_{11}f_{k1}^{*})^{q-1}v_{k0}-f_{11}^{q}f_{k2}^{*}-f_{k1}^{*q}f_{12}=0$ for $2\leqslant k\leqslant d$.
This implies that every $f_{k2}^{*}\in \F_{q}[\A][f^{-1}]$, as required.
}\end{exam}

\section{Symplectic, Unitary and Orthogonal Groups}\label{Section4}
\setcounter{equation}{0}
\renewcommand{\theequation}
{4.\arabic{equation}}
\setcounter{theorem}{0}
\renewcommand{\thetheorem}
{4.\arabic{theorem}}

For the symplectic groups,  let $W$ be an $2n$-dimensional vector space over ${\mathbb F}_{q}$. Let $K=(k_{ij})$ be a $2n\times 2n$ nonsingular alternate matrix, i.e., $k_{ij}=
-k_{ji}$ for $i\neq j$ and $k_{ii}=0$, and $Sp_{2n}({\mathbb F}_q,K)$ be the symplectic group of degree $2n$ with respect to $K$ over ${\mathbb F}_q$:
$$
Sp_{2n}({\mathbb F}_q,K)=\{T\in{\rm GL}(W)\mid TK{^tT}=K\}.
$$
For $1\leq i\leq m$, $1\leq j\leq 2n$ and $k\in{\mathbb N}^+$, we define
$$
Q_{ij}^{(k)}:=~^{t}X_{i}\cdot K\cdot X_{j}^{q^{k}}\in{\mathbb F}_{q}[mW]={\mathbb F}_{q}[x_{11},x_{12},\dots,x_{1,2n},x_{21},\dots,x_{m,2n}].
$$

For the unitary groups, suppose ${\mathbb F}_{q^2}$ has  odd characteristic. Let $W$ be an $n$-dimensional vector space over ${\mathbb F}_{q^2}$. There is an involution on ${\mathbb F}_{q^2}$: $a\mapsto \overline{a}=a^q$.
Let $H$ be an $n\times n$ nonsingular Hermitian matrix, i.e., $^t\overline{H}=H$, and $U({\mathbb F}_{q^2},H)$ be the unitary group of degree $n$ with respect to $H$
over ${\mathbb F}_{q^2}$:
$$
U({\mathbb F}_{q^2},H)=\{T\in{\rm GL}(W)\mid  TH{^t\overline{T}}=H\}.
$$
For $1\leq i\leq m$, $1\leq j\leq n$ and $k\in{\mathbb N}$, we define
$$
H_{ij}^{(k)}:=~^{t}X_{i}\cdot H\cdot X_{j}^{q^{2k+1}}\in{\mathbb F}_{q^2}[mW]={\mathbb F}_{q^2}[x_{11},x_{12},\dots,x_{1n},x_{21},\dots,x_{mn}].
$$

For the orthogonal groups, suppose ${\mathbb F}_{q}$ is of odd characteristic and $W$ is an $n$-dimensional vector space over ${\mathbb F}_{q}$.
 Let $A$ be an $n\times n$ nonsingular symmetric matrix  and $O_{n}({\mathbb F}_q,A)$ be the orthogonal group of degree $n$ with respect to $A$ over ${\mathbb F}_q$:
$$
O_{n}({\mathbb F}_q,A)=\{T\in{\rm GL}(W)\mid TA{^tT}=A\}.
$$
For $1\leq i\leq m$, $1\leq j\leq n$ and $k\in{\mathbb N}$, we define
$$
P_{ij}^{(k)}:=~^{t}X_{i}\cdot A\cdot X_{j}^{q^{k}}\in{\mathbb F}_{q}[mW]={\mathbb F}_{q}[x_{11},x_{12},\dots,x_{1n},x_{21},\dots,x_{mn}].
$$

The purpose of this section is to find a minimal generating set of polynomial invariants for ${\mathbb F}_{q}(mW)^{Sp_{2n}({\mathbb F}_q,K)}, {\mathbb F}_{q^2}(mW)^{U({\mathbb F}_{q^2},H)}$
and  ${\mathbb F}_{q}(mW)^{O_{n}({\mathbb F}_q,A)}$,
for $m\in{\mathbb N}^{+}$.

We suppose $W_{i}\cong W$ for $1\leq i\leq m$,  and we identify  ${\mathbb F}_{q}(\oplus_{i=1}^{m}W_{i})$ with
${\mathbb F}_{q}(mW)$. Clearly, every
$Q_{ij}^{(k)}\in {\mathbb F}_{q}(mW)^{Sp_{2n}({\mathbb F}_q,K)}$, $H_{ij}^{(k)}\in{\mathbb F}_{q^2}(mW)^{U({\mathbb F}_{q^2},H)}$ and  $P_{ij}^{(k)}\in{\mathbb F}_{q}(mW)^{O_{n}({\mathbb F}_q,A)}$.
On the other hand, for any $\sigma\in {\rm GL}(W)$, by Chu \cite[Lemma]{Chu1997}, $\sigma$ belongs to $Sp_{2n}({\mathbb F}_q,K)$,
$U({\mathbb F}_{q^2},H)$  and  $O_{n}({\mathbb F}_q,A)$ respectively if $\sigma$ fixes $Q_{ii}^{(k)}$, $H_{ii}^{(k)}$
and  $P_{ii}^{(k)}$ respectively  for some $k,i\geq 1$. Furthermore, by \cite[Theorem]{Chu1997}, ${\mathbb F}_{q}(W)^{Sp_{2n}({\mathbb F}_q,K)}$,
${\mathbb F}_{q^2}(W)^{U({\mathbb F}_{q^2},H)}$
 and  ${\mathbb F}_{q}(W)^{O_{n}({\mathbb F}_q,A)}$ are all purely transcendental:
\begin{eqnarray*}
{\mathbb F}_{q}(W)^{Sp_{2n}({\mathbb F}_q,K)}&=&{\mathbb F}_{q}(Q_{ii}^{(1)},Q_{ii}^{(2)},\ldots,Q_{ii}^{(2n)})\\
{\mathbb F}_{q^2}(W)^{U({\mathbb F}_{q^2},H)}&=&{\mathbb F}_{q^2}(H_{ii}^{(0)},H_{ii}^{(1)},\ldots,H_{ii}^{(n-1)})\\
{\mathbb F}_{q}(W)^{O_{n}({\mathbb F}_q,A)}&=&{\mathbb F}_{q}(P_{ii}^{(0)},P_{ii}^{(1)},\ldots,P_{ii}^{(n-1)}).
\end{eqnarray*}
It follows that, for example,
\begin{equation}
\label{ }
P_{ii}^{(k)}\in {\mathbb F}_{q}(P_{ii}^{(0)},P_{ii}^{(1)},\dots,P_{ii}^{(n-1)}),\textrm{ for all } k\geq n.
\end{equation}

The following is our main result in this section.

\begin{thm}\label{thm4.1}
{\rm (1)} Let $K$  be a $2n\times 2n$ nonsingular alternate matrix, $Sp_{2n}({\mathbb F}_q,K)$ be the symplectic group of degree $2n$ with respect to $K$ over ${\mathbb F}_q$ and $W$ be the standard representation of $Sp_{2n}({\mathbb F}_q,K)$. Then, for any $m\geq 1$,
$$
{\mathbb F}_{q}(mW)^{Sp_{2n}({\mathbb F}_q,K)}={\mathbb F}_{q}(Q_{i1}^{(k)}\mid 1\leq i\leq m,1\leq k\leq 2n).
$$
{\rm (2)}  Suppose ${\mathbb F}_{q}$ has odd characteristic. Let $H$ be an $n\times n$ nonsingular Hermitian matrix, $U({\mathbb F}_{q^2},H)$ be the unitary group of degree $n$ with respect to $H$
over ${\mathbb F}_{q^2}$ and $W$ be the standard representation of $U({\mathbb F}_{q^2},H)$. Then, for any $m\geq 1$,
$$
{\mathbb F}_{q^2}(mW)^{U({\mathbb F}_{q^2},H)}={\mathbb F}_{q^2}(H_{i1}^{(k)}\mid 1\leq i\leq m,0\leq k\leq n-1).
$$
{\rm (3)}  Suppose ${\mathbb F}_{q}$ is of odd characteristic.
 Let $A$ be an $n\times n$ nonsingular symmetric matrix, $O_{n}({\mathbb F}_q,A)$ be the orthogonal group of degree $n$ with respect to $A$ over ${\mathbb F}_q$
 and $W$ be the standard representation of $O_{n}({\mathbb F}_q,A)$.
Then
 $$
{\mathbb F}_{q}(mW)^{O_{n}({\mathbb F}_q,A)}={\mathbb F}_{q}(P_{i1}^{(k)}\mid 1\leq i\leq m,0\leq k\leq n-1).
$$
\end{thm}

\begin{proof} Let us prove (3) firstly.
We suppose $m\geq n$. Let $E={\mathbb F}_{q}(P_{i1}^{(k)}\mid 1\leq i\leq m,0\leq k\leq n-1)$ and $E^*=E(P_{1j}^{(1)}\mid 2\leq j\leq n)$. Note that $A$ is symmetric. For $1\leq j\leq n$ and $1\leq i\leq m$, we have
\begin{eqnarray*}
 \frac{\ell_{ij}}{\ell_{0}}& = & \frac{\det((X_{1},\dots,X_{j-1},\widetilde{X_{i}},X_{j+1},\dots,X_{n}))}{\det(X_{1},X_{2},\dots,X_{n})} \\
& = &  \frac{\det\left(^{t}(X_{1},\dots,X_{j-1},\widetilde{X_{i}},X_{j+1},\dots,X_{n})\cdot A\cdot (X_1,X_1^q,\dots, X_1^{q^{n-1}})\right)}{\det\left(^{t}(X_{1},X_{2},\dots,X_{n})\cdot A\cdot (X_1,X_1^q,\dots, X_1^{q^{n-1}})\right)} \\
&\in& E^*,
\end{eqnarray*}
where $\widetilde{X_{i}}=\begin{cases}
X_{i}^{q}, & \text{if } i\leq n; \\
X_{i}, & \text{if } n+1\leq i\leq m
\end{cases}$.
By Theorem \ref{Steinberg1}, we see that ${\mathbb F}_{q}(mW)^{{\rm GL}(W)}\subseteq E^*$, which both are contained in ${\mathbb F}_{q}(mW)^{O_{n}({\mathbb F}_q,A)}$.
Note that \cite[Lemma]{Chu1997} shows that for any $\sigma\in {\rm GL}(W)$, $\sigma\in O_{n}({\mathbb F}_q,A)$ if
$\sigma(P_{11}^{(k)})=P_{11}^{(k)}$ for some $k\geq 1$. Applying Galois theory, we have
${\mathbb F}_{q}(mW)^{O_{n}({\mathbb F}_q,A)}=E^*$.
Thus, to show ${\mathbb F}_{q}(mW)^{O_{n}({\mathbb F}_q,A)}=E$, it suffices to show that $E=E^*$, i.e., $P_{1j}^{(1)}\in E$ for $2\leq j\leq n.$
Let $\widetilde{X}=(X_1,X_1^q,\dots, X_1^{q^{n-2}},X_j)$ for a fixed $j\in\{2,3,\dots,n\}$. Then
\begin{eqnarray}\label{eq4.2}
& &\left(\det(^{t}\widetilde{X}\cdot A\cdot (X_1,X_1^q,\dots, X_1^{q^{n-1}}))\right)^q\nonumber\\
&=&\left(\det(A)\det(^{t}(X_1,X_1^q,\dots, X_1^{q^{n-1}})\cdot A\cdot (X_1,X_1^q,\dots, X_1^{q^{n-1}}))\right)^{\frac{q-1}{2}}\cdot\\
&&\det\left(^{t}\widetilde{X}^{(q)}\cdot A\cdot (X_1,X_1^q,\dots, X_1^{q^{n-1}})\right),\nonumber
\end{eqnarray}
where $\widetilde{X}^{(q)}=(X_1^q,X_1^{q^2},\dots, X_1^{q^{n-1}},X_j^q)$.
Clearly, $\det(^{t}\widetilde{X}\cdot A\cdot (X_1,X_1^q,\dots, X_1^{q^{n-1}}))$ and 
$$\det(^{t}(X_1,X_1^q,\dots,X_1^{q^{n-1}})\cdot A\cdot(X_1,X_1^q,\dots, X_1^{q^{n-1}}))$$ both belong to $E$, and furthermore,
\begin{equation}
\label{eq4.3}
\det\left(^{t}\widetilde{X}^{(q)}\cdot A\cdot(X_1,X_1^q,\dots, X_1^{q^{n-1}})\right)=
\det\begin{pmatrix}
P_{11}^{(1)}&P_{11}^{(0)q}&\cdots&P_{11}^{(n-2)q}\\
\vdots&\vdots&\cdots&\vdots\\
P_{11}^{(n-1)}&P_{11}^{(n-2)q}&\cdots&P_{11}^{(0)q^{n-1}}\\
P_{1j}^{(1)}&P_{j1}^{(0)q}&\cdots&P_{j1}^{(n-2)q}
\end{pmatrix}.
\end{equation}
Combining (\ref{eq4.2}) and (\ref{eq4.3}), we see that $P_{1j}^{(1)}\in E$ for all $2\leq j\leq n.$
This confirms the statement for the case $m\geq n$.

Together with the constructions of $\ell_{ij}$ and $\ell_{0}$ in Section 2, an analogous argument can apply for the remaining case $m<n$. This completes the proof of (3).

(1) and (2). From the proof of (3),  it is enough to show that every $\frac{\ell_{ij}}{\ell_0}$ belongs to
${\mathbb F}_{q}(Q_{i1}^{(k)}\mid 1\leq i\leq m,1\leq k\leq 2n)$ and ${\mathbb F}_{q^2}(H_{i1}^{(k)}\mid 1\leq i\leq m,0\leq k\leq n-1)$. However, they are true as we will see. In the case (1),
\begin{eqnarray*}
 \frac{\ell_{ij}}{\ell_{0}}& = & \frac{\det((X_{1},\dots,X_{j-1},\widetilde{X_{i}},X_{j+1},\dots,X_{2n}))}{\det(X_{1},X_{2},\dots,X_{2n})} \\
& = &  \frac{\det\left(^{t}(X_{1},\dots,X_{j-1},\widetilde{X_{i}},X_{j+1},\dots,X_{2n})\cdot K\cdot (X_1^q,X_1^{q^2},\dots, X_1^{q^{2n}})\right)}{\det\left(^{t}(X_{1},X_{2},\dots,X_{2n})\cdot K\cdot (X_1^q,X_1^{q^2},\dots, X_1^{q^{2n}})\right)},
\end{eqnarray*}
where $\widetilde{X_{i}}=\begin{cases}
X_{i}^{q}, & \text{if } i\leq n; \\
X_{i}, & \text{if } n+1\leq i\leq m
\end{cases}$. The statement follows from
$$
^tX_i^q\cdot K\cdot  (X_1^q,X_1^{q^2},\dots, X_1^{q^{2n}})=(0,(Q_{i1}^{(1)})^q,\ldots,(Q_{i1}^{(2n-1)})^q).
$$

In the case (2),
\begin{eqnarray*}
 \frac{\ell_{ij}}{\ell_{0}}& = & \frac{\det((X_{1},\dots,X_{j-1},\widetilde{X_{i}},X_{j+1},\dots,X_{n}))}{\det(X_{1},X_{2},\dots,X_{n})} \\
& = &  \frac{\det\left(^{t}(X_{1},\dots,X_{j-1},\widetilde{X_{i}},X_{j+1},\dots,X_{n})\cdot H\cdot (X_1^q,X_1^{q^3},\dots, X_1^{q^{2n-1}})\right)}{\det\left(^{t}(X_{1},X_{2},\dots,X_{n})\cdot H\cdot (X_1^q,X_1^{q^3},\dots, X_1^{q^{2n-1}})\right)},
\end{eqnarray*}
where $\widetilde{X_{i}}=\begin{cases}
X_{i}^{q^2}, & \text{if } i\leq n; \\
X_{i}, & \text{if } n+1\leq i\leq m
\end{cases}$. We have to show that $~^tX_j^{q^2}\cdot H\cdot X_1^q\in{\mathbb F}_{q^2}(H_{i1}^{(k)}\mid 1\leq i\leq m,0\leq k\leq n-1)$ for $2\leq j\leq n$.
Similar to (\ref{eq4.2}), one has
\begin{eqnarray}\label{eq4.4}
& &\left(\det(^{t}\widetilde{X}\cdot H\cdot (X_1^q,X_1^{q^3},\dots, X_1^{q^{2n-1}}))\right)^{q^2}\nonumber\\
&=&\left(\det(H)\det(^{t}(X_1^q,X_1^{q^3},\dots, X_1^{q^{2n-1}})\cdot H\cdot (X_1^q,X_1^{q^3},\dots, X_1^{q^{2n-1}}))\right)^{\frac{q^2-1}{2}}\cdot\\
&&\det\left(^{t}\widetilde{X}^{(q^2)}\cdot H\cdot (X_1^q,X_1^{q^3},\dots, X_1^{q^{2n-1}})\right),\nonumber
\end{eqnarray}
where $\widetilde{X}=(X_1^{q},X_1^{q^3},\dots, X_1^{q^{2n-3}},X_j)$. Then the same argument gives the required conclusion for $~^tX_j^{q^2}\cdot H\cdot X_1^q$.
\end{proof}

Note that Tang-Wan \cite{TW2006} has already found a generating set of polynomial invariants for the invariant field
$\F_{q}(W)^{O(W)}$ where $O(W)$ denotes the orthogonal group of even characteristic.

We conclude this paper with the following remark.

\begin{rem}{\rm
Suppose $G$ denotes any finite classical group over a finite field ${\mathbb F}_{q}$ with the standard representation $W$. Let $W^{*}$ denote the dual space of $W$ and $m,d\in{\mathbb N}$.
Our approach, together with the method appeared  in \cite{CW2019}, might be workable to find homogeneous polynomial invariants
$f_{1},f_{2},\dots,f_{(m+d)n}\in{\mathbb F}_{q}[mW\oplus dW^{*}]^{G}$ such that
$ {\mathbb F}_{q}(mW\oplus dW^{*})^{G}={\mathbb F}_{q}(f_{1},f_{2},\dots,f_{(m+d)n})$.
}\end{rem}

\section*{Acknowledgments} 
The first author was supported by SAFEA (P182009020) and  FRFCU (2412017FZ001).
The first author would like to thank David L. Wehlau for many interesting conversations on \cite{Ste1987} and Gregor Kemper for helpful comments on the rationality problem of vector invariant fields.
The second author was supported by National Natural  Science Foundation of China (11471234).


\end{document}